\documentclass[final]{siamltex}

\usepackage{showkeys}
\usepackage{color}

\def\f{\frac}
\def\ee{\textbf{e}}

\newcommand{\eps}{\ensuremath{\epsilon}}

\newcommand{\bB}{\ensuremath{\mathcal{B}}}

\newcommand{\fF}{\ensuremath{\mathcal{F}}}

\newcommand{\tT}{\ensuremath{\mathcal{T}}}

\newcommand{\ltn}{\ensuremath{\left| \! \left| \! \left|}}
\newcommand{\rtn}{\ensuremath{\right| \! \right| \! \right|}}

\title{Stochastic lattice dynamical systems with fractional noise
\thanks{H. Bessaih: Supported by NSF grant  DMS-1418838. M.J. Garrido-Atienza, X. Han and B. Schmalfu{\ss}: Partially supported by the European Funds MTM2015-63723-P (MINECO/FEDER, EU).}}

\author{Hakima Bessaih\thanks{University of Wyoming,
Laramie, WY 82071-3036, USA ({\tt bessaih@uwyo.edu}).}
\and
Mar\'{\i}a J. Garrido-Atienza\thanks{Dpto. Ecuaciones Diferenciales y An\'alisis Num\'erico, Universidad de Sevilla, Apdo. de Correos 1160, 41080-Sevilla, Spain, ({\tt mgarrido@us.es}).}
\and
Xiaoying Han\thanks{Department of Mathematics \& Statistics, Auburn University, Auburn, AL 36849, USA ({\tt xzh0003@auburn.edu}).}
\and
Bj{\"o}rn Schmalfu{\ss }\thanks{Institut f\"{u}r Mathematik,
Ernst Abbe Platz 2, 07737, Jena, Germany,  ({\tt bjoern.schmalfuss@uni-jena.de}).}}

\begin{document}

\maketitle

\begin{abstract}
This article is devoted to study stochastic lattice dynamical systems driven by a fractional Brownian motion with Hurst parameter $H\in(1/2,1)$.  First of all, we investigate the existence and uniqueness of pathwise mild solutions to such systems by the Young integration setting and prove that the solution generates a random dynamical system.  Further, we analyze the exponential stability of the trivial solution.
\end{abstract}



\pagestyle{myheadings}
\thispagestyle{plain}
\markboth{H. Bessaih, M.J. Garrido-Atienza, X. Han and B. Schmalfu\ss}{Stochastic lattice dynamical systems with fBm}

\section{Introduction}
Lattice dynamical systems arise in a wide range of applications where the spatial structure has a discrete character,  such as  image processing \cite{Chua1, Chua3, Chua2, Chua4}, pattern recognition \cite{CM-P1, CM-PV},  and  chemical reaction theory \cite{Erneux, Kapval, Laplante}.    In particular, lattice  systems have been used in biological systems to  describe the dynamics of pulses in myelinated axons where the membrane is excitable only at
spatially discrete sites  \cite{Bell, BellCosner,  Keener1, Keener2, Rashevsky,  Scott}.   Lattice systems have also been used in fluid dynamics to describe the fluid turbulence in shell models (see, e.g. \cite{BGSch, YO87}).  For some cases, lattice dynamical systems arise as discretization of partial differential equations, while they can be interpreted as ordinary differential equations in Banach spaces which are often simpler to analyze.
\vskip0.1in

Random effects arise naturally in these models to take into account the uncertainty (see, e.g. \cite{Gerstner}).  In this paper, we will consider  the following  stochastic lattice dynamical system (SLDS) with a diffusive adjacent neighborhood interaction, a dissipative nonlinear reaction term  and a fractional Brownian motion (fBm) at each node:
\begin{equation}
d u_i(t)  = \left(\nu(u_{i-1}-2u_i + u_{i+1}) -\lambda u_i + f_i (u_i)  \right) d t + \sigma_i h_i(u_i) d B_i^H(t),  \, i \in Z , \label{eq0b}
\end{equation}
with initial condition $u_i(0)$, where $Z$ denotes the integers set,  $\nu$ and $\lambda$ are positive constants,  $u_i,\, \sigma_i \in R$,   each $B_i^H(t)$ is a one-dimensional two-sided fBm with Hurst parameter $H\in (1/2,1)$, and $f_i$ and $h_i$ are smooth functions satisfying proper conditions.

\vskip0.1in

On the other hand, the theory of random dynamical systems (RDSs) has been developed by L. Arnold (see the monograph \cite{Arnold}) and his collaborators.  Thanks to this theory, we can study the stability behavior of solutions of differential equations containing a general type of noise, in terms of random attractors and their dimensions, random fixed points, random inertial, stable or unstable manifolds, and Lyapunov exponents. Finite dimensional It\^o equations with sufficiently smooth coefficients generate RDSs.  This assertion follows from the flow property
generated by the It\^o equation, due to Kolmogorov's  theorem for a H{\"o}lder  continuous version of a random field with finitely many parameters, see \cite{Kun90}.
 However this method fails for infinite dimensional stochastic equations, i.e., for systems with infinitely many parameters, and in particular for SLDSs.   To justify the flow property  or the generation of an RDS by a SLDS,  a special transform technique has been used in the literature.   Such a transform reformulates a SLDS to a pathwise random differential equation, by using Ornstein-Uhlenbeck processes.
But this technique applies only to SLDSs with random perturbations   given by either an additive white noise   $\sigma_i d B^{1/2}_i(t)$ or a simple multiplicative white noise  $\sigma_i u_i  dB^{1/2}_i(t)$ at each node $i \in Z$ (see \cite{bates-lisei-lu, bates, CHSV, Caraballo-Lu, han-review, HanShenZhou} and the references therein).  Nevertheless, there are some recent works where the generation of an RDS is established for the solution of abstract stochastic differential equations and stochastic evolution equations without transformation into random systems, see \cite{GLS, GLSch, GMS08}, where $H\in (1/2,1)$, and \cite{GLSch15, GLSch16} where $H\in (1/3,1/2]$. Note that in these last two papers the case of a Brownian motion $B^{1/2}$ is considered, giving a positive answer to the rather open problem of the generation of RDSs for systems with general diffusion noise terms.

\vskip0.1in

Our main goal in this paper is to develop new techniques of stochastic analysis to analyze the dynamics of SLDSs perturbed by  general  fBms with Hurst parameter $H\in (1/2,1)$.  In probability theory, an fBm is a centered Gau{\ss}--process with a special covariance function determined by the Hurst parameter $H\in (0,1)$.    For $H=1/2$,
$B^{1/2}$ is the Brownian motion where the generalized temporal derivative is the white noise.  For $H\not =1/2$, $B^H$ is not a semi-martingale and, as a consequence,  classical techniques of Stochastic Analysis are not applicable.  In particular, the fBm with a Hurst parameter $H \in (1/2,1)$   enjoys the property of a long range memory, which roughly implies that the decay of stochastic dependence with respect to the past is only sub-exponentially slow.     This long-range dependence property of the fBm makes it a realistic choice of noise for problems with long memory in the applied sciences.

\vskip0.1in

In this paper, we prove the existence of a unique mild solution for system (\ref{eq0b}) and analyze the exponential stability of the trivial solution.
The existence of the unique solution for a fixed initial condition relies on a fixed point argument, based on nice estimates satisfied by the stochastic integral with an fBm as integrator.  Further, we prove that the trivial solution of the SLDS is exponentially stable, namely, assuming that zero is a solution of the SLDS,  then any other solution converges to the trivial solution exponentially fast, provided that the corresponding initial data belongs to a random neighborhood of zero.   Since we do not transform the underlying SLDS into a random equation, the norm of any non-trivial solution depends on the magnitude of the norm of the noisy input.   Therefore to obtain stability we develop a cut--off argument, by which  the functions appearing in the SLDS only need to be defined in a small time interval $[-\delta,\delta]$.    This brings up the idea of considering the composition of the functions defined locally with a cut--off--like function depending on a random variable $\hat R$.  With these compositions, we construct a sequence $(u^n)_{n\in N}$ such that each element $u^n$ is a solution of a modified SLDS on $[0,1]$ driven by a path of the fBm depending also on $n$.    It is easily conceived that we will require $u^n(0)=u^{n-1}(1)$.  The norm of each $u^n$  depends on the magnitude of the corresponding driving noise and a new random variable $R$ related to the aforementioned $\hat R$.   By a suitable choice of these random variables, we can apply a discrete Gronwall--like lemma to obtain a subexponential estimate of every element of the sequence.  Finally, temperedness comes into play in order to ensure that $(u^n)_{n\in N}$ describes the solution of our SLDS on the positive real line, and such a solution converges to the equilibrium given by the trivial solution exponentially fast.
\vskip0.1in

Recently, in \cite{GaNeSch16} the authors have considered a stochastic differential equation perturbed by a H\"older--continuous function with   H\"older exponent greater than 1/2 and have investigated the exponential stability of the trivial solution.  In this paper we extend the study of the longtime stability with exponential decay to the case of considering infinite dimensional dynamical systems. We also would like to announce the forthcoming paper \cite{GaLNeSch16}, where the authors show that the trivial solution is globally attractive, by using a technique  based on a suitable choice of stopping times that depend on the noise signal, and that shall play the key role to establish the stability results.

\vskip0.1in
 The rest of this paper is organized as follows.   In Section \ref{s1} we provide necessary preliminaries and some prior estimates to be used in the sequel,   in Section \ref{s2} we study the existence and uniqueness of pathwise solutions to (\ref{eq0b}) and in section \ref{s3} we investigate the stability of solutions to our SLDS. Section 5, the appendix, is devoted to introduce some lemmas that are used in Section 4.

\section{Preliminaries}\label{s1}
Denote by
$$\ell^2: =\{(u_i)_{i \in Z}: \sum_{i \in Z} u_i^2 < \infty\}$$ the separable Hilbert--space of square summable sequences, equipped with the norm  $$\|u\|:=\left(\sum_{i \in Z} u_i^2\right)^{\f 12}, \quad u=(u_i)_{i \in Z} \in \ell^2$$
and the inner product
$$\left\langle u, v\right\rangle = \sum_{i \in Z} u_i v_i,  \quad u=(u_i)_{i \in Z}, \, v=(v_i)_{i \in Z}     \in \ell^2.$$
Let us consider the infinite sequence $(\ee_i)_{i\in Z}$ where $\ee_i$ denotes the element in $\ell^2$ having 1 at position $i$ and 0 elsewhere.    Then $(\ee_i)_{i \in Z}$ forms a complete orthonormal basis of $\ell^2$.\\

Consider given $T_1<T_2$. Let $C^\beta([T_1,T_2];\ell^2)$ be the Banach space of  H{\"o}lder continuous functions with exponent $0<\beta<1$ having values in $\ell^2$, with norm
$$
    \|u\|_{\beta,\rho ,T_1,T_2}=\|u\|_{\infty,\rho,T_1,T_2}+\ltn u\rtn_{\beta,\rho,T_1,T_2},
$$
where $\rho\geq 0$ and
\begin{eqnarray*}
\|u\|_{\infty,\rho,T_1,T_2}&=&\sup_{s\in [T_1,T_2]}e^{-\rho(s-T_1)}\|u(s)\|,\\
\ltn u\rtn_{\beta,\rho,T_1,T_2}&=&\sup_{T_1\le s<t\le T_2}e^{-\rho(t-T_1)}\frac{\|u(t)-u(s)\|}{(t-s)^\beta}.
\end{eqnarray*}
For $\rho>0$ and $\rho=0$ the corresponding norms are equivalent. We will suppress the index $\rho$ in these notations if $\rho=0$, and we will suppress $T_1,\,T_2$ when $T_1=0$ and $T_2=1$.

Since confusion is not possible, later we will use the notation $ \|\cdot\|_{\beta,\rho ,T_1,T_2}$ to express the norms of $C^\beta([T_1,T_2];R)$ and of $C^\beta([T_1,T_2];L_2(\ell^2))$, as well.

\medskip
In order to define integrals with H\"older--continuous integrators, we next define Weyl fractional derivatives of functions on separable Hilbert spaces, see \cite{Samko}.\\

\begin{definition}\label{weyl} Let $V_1$ and $V_2$ be separable Hilbert spaces and let  $0 < \alpha < 1$.  The Weyl fractional derivatives of general measurable functions $Z:[s, t] \to V_1$ and $\omega: [s, t] \to V_2$, of order $\alpha$ and $1-\alpha$ respectively, are defined for $s<r<t$ by
\begin{eqnarray*}
D^\alpha_{s+}Z[r] &=& \f{1}{\Gamma(1-\alpha)}\left(\f{Z(r)}{(r-s)^\alpha} + \alpha \int^r_s \f{Z(r) - Z(q)}{(r - q)^{1+\alpha}}  d q \right) \in V_1,  \\
D^{1 - \alpha}_{t-} \omega_{t-}[r]&=& \f{(-1)^\alpha}{\Gamma(\alpha)}\left(\f{ \omega(r) -  \omega(t-)}{(t-r)^{1-\alpha}} + (1- \alpha) \int^t_r \f{ \omega(r) -  \omega(q)}{(q - r)^{2 - \alpha}} d q\right) \in V_2,
\end{eqnarray*}
where
$$ \omega_{t-}(r) =  \omega(r) -  \omega(t-),$$ and $ \omega(t-)$ is the left side limit of $ \omega$ at $t$.\\
\end{definition}

The next result shows that Weyl fractional derivatives are well--posed for H\"older--continuous functions with suitable H\"older exponents. The proof follows easily and therefore we omit it.\\

\begin{lemma}\label{l5}
Suppose that $Z\in C^\beta([T_1,T_2];V_1)$, $\omega\in C^{\beta^\prime}([T_1,T_2];V_2)$, $T_1\le s<t\le T_2$ and  that $1-\beta^\prime<\alpha<\beta$.
Then $D^\alpha_{s+}Z$ and $D^{1 - \alpha}_{t-} \omega_{t-}$ are well--defined.
\end{lemma}

\medskip

Let us assume for a while that $V_1=V_2=R$. Following Z{\"a}hle \cite{Zah98} we can define the fractional integral by
$$
  \int^t_s Z d \omega  = (-1)^\alpha \int^t_s D^\alpha_{s+}Z[r] D^{1-\alpha}_{t-}\omega_{t-}[r] d r.
$$

We collect some properties of this integrals, for the proof see \cite{ChGGSch12} and \cite{Zah98}.\\

\begin{lemma}\label{l1}
Let $Z,\,Z_1,\,Z_2\in C^\beta([T_1,T_2];R),\,\omega,\,\omega_1,\,\omega_2\in C^{\beta^\prime}([T_1,T_2];R)$ such that $\beta+\beta^\prime>1$. Then there exists a positive constant $C_{\beta,\beta^\prime}$ such that for $T_1\le s<t\le T_2$
$$
  \bigg|\int_s^t Zd\omega\bigg|\le C_{\beta,\beta^\prime}(1+(t-s)^{\beta})(t-s)^{\beta^\prime}\|Z\|_{\beta,T_1,T_2}\ltn\omega\rtn_{\beta^\prime,T_1,T_2}.
$$
In addition,
\begin{eqnarray*}
\int_{s}^{t}(Z_1+Z_2)d\omega&=&\int_{s}^{t}Z_1d\omega+\int_{s}^{t}Z_2d\omega,\\
\int_{s}^{t}Zd(\omega_1+\omega_2)&=&\int_{s}^{t}Zd\omega_1+\int_{s}^{t}Zd\omega_2.
\end{eqnarray*}
The integral is additive: for $\tau\in[s,t]$
$$
  \int_{s}^{t} Z d \omega =\int_{s}^{\tau} Z d \omega +\int_{\tau}^{t} Z d \omega.
$$
 Moreover, for any $\tau \in R$
\begin{equation}\label{change}
    \int_{s}^{t} Z(r) d\omega(r)=\int_{s-\tau}^{t-\tau} Z(r+\tau) d\theta_\tau \omega(r),
\end{equation}
where $\theta_\tau \omega(\cdot)=\omega(\cdot+\tau)-\omega(\tau)$. Finally, let $(\omega_n)_{n\in N}$ be a sequence converging in $C^{\beta^\prime}([T_1,T_2];R)$ to $\omega$. Then we have
$$
  \lim_{n\to\infty}\bigg\|\int_{T_1}^\cdot Zd\omega_n-\int_{T_1}^\cdot Zd\omega\bigg\|_{\beta,T_1,T_2}=0.
$$
\end{lemma}

Note that in the last expression, the integral with respect to $\omega_n$ can be interpreted in the Lebesgue sense.

\medskip

We now extend the definition of a fractional integral in $R$  to a fractional integral in the separable Hilbert space $\ell^2$, following the construction carried out recently in \cite{ChGGSch12} in a general separable Hilbert--space. To do that, consider the separable Hilbert space $L_2(\ell^2)$ of Hilbert--Schmidt operators from $\ell^2$ into $\ell^2$, with the usual norm $\|\cdot\|_{L_2(\ell^2)}$ defined by
$$\|z\|_{L_2(\ell^2)}^2=\sum_{i\in Z} \|z \ee_i\|^2,$$
for $z  \in L_2(\ell^2)$. Let $Z\in C^\beta([T_1,T_2];L_2(\ell^2))$ and $\omega\in C^{\beta^\prime}([T_1,T_2];\ell^2)$ with $\beta+\beta^\prime>1$. We define the $\ell^2$-valued integral for $T_1\le s<t\le T_2$ as
\begin{equation}\label{eq12}
  \int_{s}^{t}Zd\omega:=(-1)^\alpha
  \sum_{j\in Z}\bigg(\sum_{i\in Z}\int_s^t D_{s+}^\alpha\langle\ee_j,Z(\cdot)\ee_i\rangle[r]D_{t-}^{1-\alpha}\langle\ee_i,\omega(\cdot)\rangle_{t-}[r]dr\bigg)\ee_j,
\end{equation}
for $1-\beta^\prime< \alpha<\beta$, whose norm fulfills
$$
    \bigg\|\int_{s}^{t}Zd\omega\bigg\|\le \int_{s}^{t}\|D_{s+}^\alpha Z[r]\|_{L_2(\ell^2)}\|D_{t-}^{1-\alpha} \omega_{t-}[r]\| dr.$$
Note that in (\ref{eq12}) the integrals under the sums are one-dimensional fractional integrals. In particular, in \cite{ChGGSch12} the following result was proved:\\

\begin{theorem}\label{t1}
Suppose that $Z\in C^\beta([T_1,T_2];L_2(\ell^2))$ and $\omega\in C^{\beta^\prime}([T_1,T_2];\ell^2)$ where $\beta+\beta^\prime>1$. Then there exists $\alpha\in (0,1)$ such that $1-\beta^\prime<\alpha<\beta$ and the integral (\ref{eq12}) is well--defined. Moreover, all properties of Lemma \ref{l1} hold if we replace the $R$--norm by the $\ell^2$--norm.\\
\end{theorem}

We now consider estimates of the integral with respect to the H{\"o}lder norms depending on $\rho$.\\

\begin{lemma}\label{l1b}
Under the assumptions of Theorem \ref{t1}, for $\beta^\prime>\beta$ there exists a constant $c$ depending on $T_1,\,T_2,\,\beta,\,\beta^\prime$ such that
for $T_1 \leq s<t\leq T_2$
\begin{equation}\label{estb}
e^{-\rho t} \bigg\|\int_{s}^{t} Z d\omega\bigg\|\le c k(\rho) \|Z\|_{\beta,\rho,s,t} \ltn \omega\rtn_{\beta^\prime,s,t}(t-s)^{{\beta}},
\end{equation}
such that $\lim_{\rho\to \infty} k(\rho)=0.$
\end{lemma}
\begin{proof}
We only sketch the proof, for more details see \cite{ChGGSch12}.

First of all, it is not difficult to see that
$$\|D_{t-}^{1-\alpha}\omega_{t-}[r]\|\le c \ltn \omega\rtn_{\beta^\prime,s,t}(t-r)^{\alpha+{\beta^\prime}-1}.$$

Furthermore, since $Z\in C^{\beta}([T_1,T_2];L_2(\ell^2))$,
\begin{eqnarray*}
   e^{-\rho t}   \|D_{s+}^\alpha Z[r]\|_{L_2(\ell^2)}
    &\le&c e^{-\rho (t-r)} \bigg(e^{-\rho r} \frac{ \|Z(r)\|_{L_2(\ell^2)}}{(r-s)^\alpha}+\int_{s}^r e^{-\rho r} \frac{ \|Z(r)-Z(q)\|_{L_2(\ell^2)}}{(r-q)^{1+\alpha}}dq\bigg)\\
    &\le&  c e^{-\rho (t-r)}(1+(r-s)^\beta)\|Z\|_{\beta,\rho,s,t}(r-s)^{-\alpha}\\
     &\le&  c e^{-\rho (t-r)}\|Z\|_{\beta,\rho,s,t}(r-s)^{-\alpha}.
\end{eqnarray*}
Therefore,
\begin{eqnarray*}
e^{-\rho t}\bigg\|\int_{s}^{t} Z d\omega\bigg\| &\leq &   c  \ltn \omega\rtn_{\beta^\prime,s,t} \|Z\|_{\beta,\rho,s,t} \int_{s}^{t} e^{-\rho(t-r)}(t-r)^{\alpha+{\beta^\prime}-1}(r-s)^{-\alpha}dr\\
&\leq  & c  \ltn \omega\rtn_{\beta^\prime,s,t}\|Z\|_{\beta,\rho,s,t} (t-s)^\beta \int_{s}^{t} e^{-\rho(t-r)}(t-r)^{\alpha+{\beta^\prime-\beta}-1}(r-s)^{-\alpha}dr\\
&\leq& c k(\rho) \ltn \omega\rtn_{\beta^\prime,s,t} \|Z\|_{\beta,\rho,s,t}(t-s)^\beta,
\end{eqnarray*}
where
$$k(\rho)=\sup_{0\leq s<t \leq T}\int_{s}^{t} e^{-\rho(t-r)}(t-r)^{\alpha+{\beta^\prime-\beta}-1}(r-s)^{-\alpha}dr$$
is such that $\lim_{\rho\to \infty}k(\rho)=0$. The previous property can be stated in general as follows: given $T>0$, if $a,\,b>-1$ are such that $a+b+1>0$, then
\begin{equation}\label{rho}
k(\rho):=\sup_{0\leq s<t \leq T}\int_s^te^{-\rho (t-r)}(r-s)^{a}(t-r)^bdr,
\end{equation}
is such that $\lim_{\rho\to\infty}k(\rho)=0,$ see \cite{ChGGSch12}.
\end{proof}

\medskip

From now on $k(\rho)$ will denote a function with the above behavior no matter the exact values of the corresponding parameters $a,\,b>-1$ provided that $a+b+1>0$. Moreover, note that the constraints in Lemma \ref{l1b} imply that $\beta^\prime>1/2$.\\

As a particular case of H\"older--continuous integrator we are going to consider a fractional Brownian motion (fBm) with values in $\ell^2$ and Hurst--parameter $H>1/2$.  Consider a probability space $(\Omega,\fF, P)$. Let $(B_i^H)_{i\in Z}$ be an iid-sequence
of fBm with the same Hurst--parameter $H>1/2$ over this probability space, that is, each $B^H_i$ is a centered Gau{\ss}-process on $R$ with covariance
\begin{eqnarray*}
  {\cal R}(s,t)=\frac12(|s|^{2H}+|t|^{2H}-|t-s|^{2H})\quad  {\rm for }\, s,\,t\in R.
\end{eqnarray*}
Let $Q$ be a linear operator on $\ell^2$ such that $Q\ee_i=\sigma_i^2 \ee_i$, $\sigma=(\sigma_i)_{i\in Z}$. Hence $Q$ is a non--negative and symmetric trace--class operator. A continuous $\ell^2$-valued fBm $B^H$ with  covariance operator $Q$ and Hurst parameter $H$ is defined by
\begin{equation}\label{fBm}
  B^H(t)=\sum_{i\in Z}(\sigma_i B_i^H(t))\ee_i
\end{equation}
having covariance
\begin{eqnarray*}
  {\cal R}_Q(s,t)=\frac12Q(|s|^{2H}+|t|^{2H}-|t-s|^{2H})\quad  {\rm for }\, s,\,t\in R.
\end{eqnarray*}
In fact, since $B^H$ is a Gau\ss--process,
\begin{eqnarray*}
E\left\| B^H(t) - B^H(s)\right\|^2
&= &  \sum_{i \in Z} \sigma_i^2 E(B^H_i(t) - B^H_i(s))^2 = \sum_{i \in Z}\sigma_i^2 |t-s|^{2H} = \|\sigma\|^2 |t-s|^{2H},\\
E\left\|B^H(t) - B^H(s)\right\|^{2n} & \leq & c_n |t-s|^{2Hn}.
\end{eqnarray*}
Therefore, applying Kunita \cite{Kun90} Theorem 1.4.1, $B^H(t)$ has a continuous version and also a H\"older--continuous version with exponent $\beta^\prime<H$, see Bauer \cite{Bau096} Chapter 39. Note that $B^H(0)=0$ almost surely.

\medskip

Let $C_0(R;\ell^2)$ be the space of continuous functions on $R$ with values in $\ell^2$ which are zero at zero,
equipped with the compact open topology.  Consider the
canonical space for the fBm $(C_0(R;\ell^2),\bB(C_0(R;\ell^2)),P_H)$, where $B^H(\omega)=\omega$ and $P_H$ denotes the measure of the fBm with Hurst--parameter $H$.  On $C_0(R;\ell^2)$ we can introduce the Wiener shift $\theta$ given by the measurable flow
$$
  \theta: (R\times C_0(R,\ell^2),\bB(R)\otimes\bB(C_0(R,\ell^2)))\to (C_0(R,\ell^2),\bB(C_0(R,\ell^2)))
$$
such that
\begin{equation}\label{shift}
  \theta(t,\omega)(\cdot)=\theta_t\omega(\cdot)=\omega(\cdot+t)-\omega(t).
\end{equation}
By Mishura \cite{Mis08}, Page 8, we have that $\theta_t$ leaves $P_H$ invariant. In addition $t\to\theta_t\omega$ is continuous. Furthermore, thanks to Bauer \cite{Bau096} Chapter 39, we can also conclude that the set $C_0^{\beta^\prime}(R;\ell^2)$ of continuous functions which have a finite $\beta^\prime$--H{\"o}lder-seminorm on any compact interval and which are zero at zero has $P_H$-measure one for $\beta^\prime<H$.
This set is $\theta$-invariant.
\medskip

\section{Lattice equations driven by fractional Brownian motions}\label{s2}

Given strictly positive constants $\nu$ and $\lambda$, we consider the following SLDS with a diffusive adjacent neighborhood interaction, a dissipative nonlinear reaction term,  and an fBm $B^H_i$ at each node:
\begin{equation}\label{eq0}
d u_i(t)  = \left(\nu(u_{i-1}-2u_i + u_{i+1})  -\lambda u_i+ f_i (u_i)  \right) d t + \sigma_i h_i(u_i) d B^H_i(t),  \, i \in Z.
\end{equation}
Here $f_i$ and $h_i$ are suitable regular functions, see below. We want to rewrite this system giving it the interpretation of a stochastic evolution equation in $\ell^2$. To this end,  let $A$ be the linear bounded operator from $\ell^2$ to $\ell^2$ defined by $Au = ((Au)_i)_{i \in Z}$ where
\begin{eqnarray*}
 (Au)_i = - \nu(u_{i-1} - 2 u_i + u_{i+1}), \quad i \in Z. \label{operatorA}
\end{eqnarray*}
Notice that  $A = BB^* = B^*B$, where
\begin{eqnarray*}
  (Bu)_i = \sqrt{\nu}(u_{i+1} - u_i), \quad (B^*u)_i = \sqrt{\nu}(u_{i-1} - u_i)
  \end{eqnarray*}
and hence
\begin{eqnarray*}
  \left\langle Au, u \right\rangle \geq 0, \quad \forall u \in \ell^2.
\end{eqnarray*}
Let us consider the linear bounded operator $A_\lambda:\ell^2\to \ell^2$ given by
\begin{equation}\label{eq1}
  A_\lambda u=Au+\lambda u.
\end{equation}
Then
\begin{eqnarray*}
  \left\langle A_\lambda u, u \right\rangle \geq \lambda \|u\|^2, \quad \forall u \in \ell^2,
\end{eqnarray*}
hence $-A_\lambda$ is a negative defined and bounded operator, thus it generates a uniformly continuous (semi)group $S_\lambda:=e^{-A_\lambda t}$ on $\ell^2$, for which the following estimates hold true:\\

\begin{lemma}\label{l3}
The uniformly continuous semigroup $S_\lambda$ is also exponentially stable, that is,  for $t\geq 0$ we have
\begin{eqnarray}\label{est}
  \|S_\lambda(t)\|_{L(\ell^2)}\le e^{-\lambda t}.
\end{eqnarray}
In addition, for $0\le s \leq t$
\begin{eqnarray}
\|S_\lambda(t-s)-{\rm id} \|_{L(\ell^2)}& \leq& \|A_\lambda\|(t-s),\nonumber \\[-1.5ex]
\label{semi}\\[-1.5ex]
\|S_\lambda(t)-S_\lambda(s)\|_{L(\ell^2)}& \leq &\|A_\lambda\| (t-s) e^{-\lambda s},\nonumber
\end{eqnarray}
where, for the sake of presentation, $\|A_\lambda\|$ represents $\|A_\lambda\|_{L(\ell^2)}$ ($L(\ell^2)$ denotes the space of linear continuous operator from $\ell^2$ into itself).\\
\end{lemma}

The proof of the first property is a direct consequence of the energy inequality, while the two last estimates follow easily by the mean value theorem. As straightforward results, we also obtain that  for $0<s<t$,
\begin{eqnarray}\label{eq3}
\ltn S_\lambda(t-\cdot)\rtn_{\beta,0,t}&=&\sup_{0\leq r_1<r_2\le t}\frac{\|S_\lambda(t-r_2)-S_\lambda(t-r_1)\|_{L(\ell^2)}}{(r_2-r_1)^\beta}\le \|A_\lambda\| t^{1-\beta},
\end{eqnarray}
and
\begin{eqnarray}
& &\ltn S_\lambda(t-\cdot)-S_\lambda(s-\cdot)\rtn_{\beta,0,s} \nonumber \\
&=& \sup_{0\leq r_1<r_2\leq s}\frac{\|(S_\lambda(t-s)-{\rm id})(S_\lambda(s-r_2)-S_\lambda(s-r_1))\|_{L(\ell^2)}}{(r_2-r_1)^\beta}\label{sem1} \\
&\leq& \|A_\lambda\|^2 (t-s) s^{1-\beta}.\nonumber
\end{eqnarray}

Now we formulate the assumptions for the functions $f_i$ and $g_i$. Indeed, for the sake the completeness, we present now all the standing assumptions needed in this section:
\begin{itemize}
	\item [(\textbf{A1})] The process $\omega$ is a (canonical) continuous fBm with values in $\ell^2$, with covariance $Q$ and with Hurst--parameter $H$ given by (\ref{fBm}). In particular, we have parameters $\f12 < \beta < \beta^\prime < H \quad \mbox{and} \quad 1 - \beta^\prime < \alpha < \beta.$
\item [(\textbf{A2})]
Let $A_\lambda$ be the operator defined by (\ref{eq1}), and $S_\lambda$ the exponentially stable and uniformly continuous semigroup generates by $-A_\lambda$.
\item [(\textbf{A3})] $f_i \in C^{1}(R, R)$, $\sum_{i\in Z}f_i(0)^2 <\infty$, and there exists a constant $D_f \geq 0$ such that
	\begin{eqnarray*}
	\left|f^\prime_i(\zeta)\right| \leq D_f,\quad \zeta \in R, \, i \in Z.
	\end{eqnarray*}
\item [(\textbf{A4})] $h_i \in C^2(R, R)$, $\sum_{i\in Z}h_i(0)^2 <\infty$, and there exist constants $D_h, \,M_h\geq 0$ such that
\begin{eqnarray*}
  \left|h^\prime_i(\zeta)\right| \leq D_h, \quad \left|h^{\prime\prime}_i(\zeta)\right|\leq M_h,\quad \zeta \in R, \, i \in Z.
\end{eqnarray*}
\end{itemize}

Let $u=(u_i)_{i\in Z}$ be an element of $\ell^2$. Then
(\textbf{A3}) allows us to define the operator
\begin{equation}\label{f}
  f:\ell^2\to \ell^2,\quad f(u):=(f_i(u_i))_{i\in Z}.
\end{equation}
Thanks to (\textbf{A4}) we can also define the operator $h(u)\in L(\ell^2)$ by
\begin{equation}\label{h}
  h(u)v=(h_i(u_i)v_i)_{i\in Z}\in \ell^2.
\end{equation}

That $f$ and $h$ are well--posed is proved in the next result, as well as their main regularity properties.\\

\begin{lemma}\label{l2}
a) The operator  $f:\ell^2\to \ell^2$ given by (\ref{f}) is well--defined and is Lipschitz continuous with Lipschitz constant $D_f$.

b) The operator  $\ell^2\ni u\mapsto h(u)\in L_2(\ell^2)$ given by (\ref{h}) is well--defined and continuously differentiable. Moreover, both $h$ and its first derivative $Dh$ are Lipschitz--continuous with Lipschitz constants $D_h$ and $M_h$, respectively. Furthermore, for $u,\,v,\,w,\,z \in \ell^2$ the following property holds true:
\begin{eqnarray}
 \| h(u)-h(v) & -& (h(w)-h(z))\|_{L_2(\ell^2)}\le \sqrt{2} D_h\|u-v-(w-z)\|\nonumber \\[-1.5ex]
\label{MN}\\[-1.5ex]
 &+ &2M_h\|u-w\|(\|u-v\|+\|w-z\|).\nonumber
\end{eqnarray}
\end{lemma}
\begin{proof}

a) Thanks to the definition of $f$, for $u=(u_i)_{i\in Z}\in \ell^2$ we have
\begin{eqnarray*}
  \|f(u)\|^2\le 2\sum_{i\in Z}f_i(0)^2+2D_f^2\|u\|^2<\infty,
\end{eqnarray*}
hence it is well-posed. Furthermore, $f$ is Lipschitz--continuous:  for $v=(v_i)_{i\in Z}\in \ell^2$ we obtain
\begin{eqnarray*}
  \|f(u)-f(v)\|^2=\sum_{i\in Z}|f_i(u_i)-f_i(v_i)|^2\le D_f^2\sum_{i\in Z}|u_i-v_i|^2=D_f^2\|u-v\|^2.
\end{eqnarray*}

b) The operator $h$ is well-defined as a Hilbert--Schmidt--operator, since
\begin{eqnarray*}
\|h(u)\|^2_{L_2(\ell^2)}&=\sum_{i\in Z} \|h(u)\ee_i\|^2=\sum_{i,j\in Z} |(h(u)\ee_i)_j|^2=\sum_{i\in Z} |h_i(u_i)|^2\\
&\le 2\sum_{i\in Z}h_i(0)^2+2D_h^2\|u\|^2<\infty.
\end{eqnarray*}
Moreover, in a similar way as we have proceed for the operator $f$, $h$ is Lipschitz-continuous:
\begin{eqnarray*}
\|h(u)-h(v)\|^2_{L_2(\ell^2)}\leq D_h^2\|u-v\|^2, \quad  {\rm for }\,  u,\,v \in \ell^2.
\end{eqnarray*}
Regarding the derivative, we have that $Dh:\ell^2 \mapsto L(\ell^2,L_2(\ell^2))$ is defined for $u,\,v,\,w\in \ell^2$  by
\begin{eqnarray*}
&(Dh(u)v)w=(h_i^\prime(u_i)v_iw_i)_{i\in Z}.
\end{eqnarray*}
In fact,
\begin{eqnarray*}
  \|h(u+v)&-h(u)-Dh(u)v\|_{L_2(\ell^2)}^2=\sum_{i\in Z}|h_i(u_i+v_i)-h_i(u_i)-h_i^\prime(u_i)v_i|^2\\
  \le& \frac14
  \sum_{i\in Z}|h_i^{\prime\prime}(\tilde u_i)|^2v_i^4\le \frac14 M_h^2\sum_{i\in Z}v_i^4
  \le  \frac14 M_h^2(\sum_{i\in Z}v_i^2)^2 \le  \frac14 M_h^2\|v\|^4,
\end{eqnarray*}
where $\tilde u_i$ is an intermediate element between $u_i$ and $u_i+v_i$.
This derivative is bounded in the space $L(\ell^2,L_2(\ell^2))$ since
\begin{eqnarray*}
  \|Dh(u)\|_{L(\ell^2,L_2(\ell^2))}^2=\sup_{\|z\|=1} \sum_{i\in Z} |h_i^\prime(u_i)z_i|^2\leq D_h^2,
  \end{eqnarray*}
and furthermore $Dh$ is Lipschitz--continuous:
\begin{eqnarray*}
  \|Dh(u) &-Dh(v)\|_{L(\ell^2,L_2(\ell^2))}^2=\sup_{\|z\|=1} \sum_{i\in Z}|h_i^\prime(u_i)z_i-h_i^\prime(v_i)z_i|^2\\
  &\leq \sup_{\|z\|=1} \sum_{i\in Z}|h_i^{\prime \prime}(\tilde u_i)(u_i-v_i)z_i|^2 \leq M_h^2\|u-v\|^2.
\end{eqnarray*}
Finally, property (\ref{MN}) follows by Nualart and Rascanu \cite{NuaRas02} Lemma 7.1. Indeed, in virtue of the Lispchitz continuity of any $h_i$ and $h_i^\prime$ we obtain
\begin{eqnarray*}
 \| & h(u)-h(v) -(h(w)-h(z))\|^2_{L_2(\ell^2)}=\sum_{i\in Z}|h_i(u_i)-h_i(v_i)-(h_i(w_i)-h_i(z_i))|^2\\
 &\leq \sum_{i\in Z}(2D_h^2|u_i-v_i-(w_i-z_i)|^2+4M_h^2|u_i-w_i|^2(|u_i-v_i|^2+|w_i-z_i|^2))\\
 &\leq 2 D_h^2\|u-v-(w-z)\|^2+ 4M_h^2\|u-w\|^2(\|u-v\|^2+\|w-z\|^2).
\end{eqnarray*}
\end{proof}

Hence, we can reformulate the system of equations given by (\ref{eq0}) as the following evolution equation with values in $\ell^2$:
$$du(t)=(-A_\lambda u(t)+f(u(t))) dt+h(u(t))d \omega(t),$$
where $A_\lambda$ has been defined by (\ref{eq1}), and $f$ and $h$ by (\ref{f}) and (\ref{h}), respectively. The sequence $u(t)=(u_i(t))_{i\in Z}$ is such that $u_i$ fulfills (\ref{eq0}) for each $i\in Z$.
Since our stability considerations will be based on the exponential stability of $S_\lambda$, we look for a mild solution of the above equation, namely, we look for $u(t)=(u_i(t))_{i\in Z} \in \ell^2$ solution of the operator equation
\begin{equation}\label{eq2}
  u(t)=S_\lambda(t)x+\int_0^tS_\lambda(t-r)f(u(r))dr+\int_0^tS_\lambda(t-r)h(u(r))d\omega(r),
\end{equation}
where the initial condition $x\in \ell^2$. The last integral has to be interpreted as we have explained in Section \ref{s1}.\\

Next we would like to apply a fixed point argument to ensure the existence and uniqueness of a solution to (\ref{eq2}). We first present estimates of the stochastic integral appearing on the right hand side of (\ref{eq2}).\\
\begin{lemma} Under assumptions {\bf(A1)}, {\bf(A2)} and {\bf(A4)}, the stochastic integral satisfies
\begin{equation}\label{eh1}
\bigg\| \int_0^\cdot S_\lambda(\cdot-r)h(u(r))d\omega(r) \bigg\|_{\beta,\rho, 0,T} \le ck(\rho)\ltn\omega\rtn_{\beta^\prime,0,T}
\|h(u(\cdot))\|_{\beta,\rho,0,T},
\end{equation}
where $c$ may depend on $\beta$, $\beta^\prime$, $T$, $\|A_\lambda\|$, and $k(\rho)$ is given by (\ref{rho}). Furthermore,
\begin{equation}\label{eh2}
\bigg\| \int_0^\cdot S_\lambda(\cdot-r)h(u(r))d\omega(r) \bigg\|_{\infty, 0,T} \le c(1+\|A_\lambda\|)\ltn\omega\rtn_{\beta^\prime,0,T}
\|h(u(\cdot))\|_{\beta,0,T},
\end{equation}
\begin{equation}\label{eh3}
\ltn \int_0^\cdot S_\lambda(\cdot-r)h(u(r))d\omega(r) \rtn_{\beta, 0,T} \le c(1+\|A_\lambda\|)^2\ltn\omega\rtn_{\beta^\prime,0,T}
\|h(u(\cdot))\|_{\beta,0,T},
\end{equation}
where in the last two inequalities $c$ may depend on $\beta$, $\beta^\prime$ and $T$.
\end{lemma}
\begin{proof} Thanks to the additivity of the stochastic integral we can consider the following splitting
\begin{eqnarray}
 & & \int_0^t S_\lambda(t-r)h(u(r))d\omega(r)-\int_0^s S_\lambda(s-r)h(u(r))d\omega(r) \nonumber \\[-1.5ex]
\label{eq22}\\[-1.5ex]
  &=&\int_s^t S_\lambda(t-r)h(u(r))d\omega(r)+\int_0^s (S_\lambda(t-r)-S_\lambda(s-r)) h(u(r))d\omega(r).\nonumber
\end{eqnarray}
From (\ref{estb}), for $0\leq s<t\leq T$ we obtain
\begin{eqnarray*}
  &e^{-\rho t}\frac{\bigg\|\int_s^t S_\lambda(t-r)h(u(r))d\omega(r)\bigg\|}{(t-s)^\beta}\le ck(\rho)\ltn\omega\rtn_{\beta^\prime,0,T}
  \|S_\lambda(t-\cdot)h(u(\cdot))\|_{\beta,\rho,0,t}\\
  &
  e^{-\rho t}\frac{\bigg\| \int_0^s (S_\lambda(t-r)-S_\lambda(s-r)) h(u(r))d\omega(r)\bigg\|}{(t-s)^\beta}\le
  ck(\rho)\ltn\omega\rtn_{\beta^\prime,0,T}\frac{s^\beta}{(t-s)^\beta}\\
  &\qquad \qquad \qquad \qquad \qquad \qquad \qquad \qquad \times{\|(S_\lambda(t-\cdot)-S_\lambda(s-\cdot))h(u(\cdot))\|_{\beta,\rho,0,s}}.
\end{eqnarray*}

Furthermore, since for two any $\beta$--H{\"o}lder--continuous functions $l,\,g$ we easily obtain
\begin{equation}\label{pr}
  \|lg\|_{\beta,\rho,0,t}\le \|l\|_{\infty,0,t}\|g\|_{\beta,\rho,0,t}+\|g\|_{\infty,\rho,0,t}\ltn l\rtn_{\beta,0,t},
\end{equation}
by (\ref{est}) and (\ref{eq3}) we derive
\begin{eqnarray*}
  \|S_\lambda(t-\cdot)h(u(\cdot))\|_{\beta,\rho,0,t}&\le& \|S_\lambda(t-\cdot)\|_{\infty,0,t}\|h(u(\cdot))\|_{\beta,\rho,0,t}+\ltn S_\lambda(t-\cdot)\rtn_{\beta,0,t}\|h(u(\cdot))\|_{\infty,\rho,0,t}\\
  &\le&\|h(u(\cdot))\|_{\beta,\rho,0,t}+\|A_\lambda\|t^{1-\beta}\|h(u(\cdot))\|_{\infty,\rho,0,t}
  \end{eqnarray*}
  and by (\ref{semi}) and (\ref{sem1})
  \begin{eqnarray*}
   & &\|(S_\lambda(t-\cdot)-S_\lambda(s-\cdot))h(u(\cdot))\|_{\beta,\rho,0,s}\\
   & \le &
 (t-s) \|A_\lambda\|\|h(u(\cdot))\|_{\beta,\rho,0,s}+\|A_\lambda\|^2 (t-s) s^{1-\beta}\|h(u(\cdot))\|_{\infty,\rho,0,s}.
\end{eqnarray*}
Hence
\begin{eqnarray*}
\ltn \int_0^\cdot S_\lambda(\cdot-r)h(u(r))d\omega(r) \rtn_{\beta,\rho, 0,T} \le ck(\rho)\ltn\omega\rtn_{\beta^\prime,0,T}
\|h(u(\cdot))\|_{\beta,\rho,0,T}.
\end{eqnarray*}
Taking into account the way in which we have estimated the first term on the right hand side of (\ref{eq22}), we immediately obtain
\begin{eqnarray*}
 \bigg\|\int_0^\cdot S_\lambda(\cdot-r)h(u(r))d\omega(r)\bigg\|_{\infty,\rho, 0,T}\le ck(\rho)\ltn\omega\rtn_{\beta^\prime,0,T}
\|h(u(\cdot))\|_{\beta,\rho,0,T},
\end{eqnarray*}
so the proof of (\ref{eh1}) is complete.

Notice that using the $\beta$--norm, (\ref{pr}) reads as follows
\begin{eqnarray*}
  \|lg\|_{\beta,0,t}\le \|l\|_{\infty,0,t}\|g\|_{\beta,0,t}+\|g\|_{\infty,0,t}\ltn l\rtn_{\beta,0,t},
\end{eqnarray*}
hence (\ref{eh2}) is an immediate consequence of (\ref{est}) and (\ref{eq3}). In order to prove (\ref{eh3}) we can follow the same steps as at the beginning of this proof.
\end{proof}

\medskip

Now we can establish the existence of a unique mild solution to our SLDS.\\

\begin{theorem}\label{t2}
Under assumptions {\bf(A1)}--{\bf(A4)}, for every $T>0$ and $x\in \ell^2$ the problem
(\ref{eq2}) has a unique solution $u(\cdot)=u(\cdot,\omega, x) \in C^\beta([0,T];\ell^2)$.
\end{theorem}

\begin{proof}
We will show that the operator
\begin{eqnarray*}
  \tT_{x,\omega}(u)[t]=S_\lambda(t)x+\int_0^tS_\lambda(t-r)f(u(r))dr+\int_0^tS_\lambda(t-r)h(u(r))d\omega(r),
\end{eqnarray*}
where $t\in [0,T]$, has a unique fixed point in $C^\beta([0,T];\ell^2)$ by applying the Banach fixed point theorem. To this end, first of all we show that there
exists a closed centered ball with respect to the norm $\|\cdot\|_{\beta,\rho,0,T}$ which is mapped by $\tT_{x,\omega}$ into itself.
For the first term, in virtue of (\ref{semi}),
\begin{eqnarray*}
  \|S_\lambda(\cdot)x\|_{\beta,\rho,0,T}\le (1+\|A_\lambda\|T^{1-\beta})\|x\|.
\end{eqnarray*}
For the Lebesgue integral of  $\tT_{x,\omega}$ we obtain

\begin{eqnarray}
\bigg\|\int_0^\cdot S_\lambda(\cdot-r)f(u(r))dr\bigg\|_{\beta,\rho,0,T}&\le& \sup_{t\in[0,T]}e^{-\rho t}\bigg\|\int_0^tS_\lambda(t-r)
 f(u(r))dr\bigg\| \nonumber \\
 &+&
 \sup_{0\le s<t\le T}e^{-\rho t}\frac{\bigg\|\int_s^tS_\lambda(t-r)f(u(r))dr\bigg\|}{(t-s)^\beta} \nonumber\\[-1.5ex]
 \label{eq4}\\[-1.5ex]
 &+&
 \sup_{0\le s<t\le T}e^{-\rho t}\frac{\bigg\|\int_0^s(S_\lambda(t-r)-S_\lambda(s-r))f(u(r))dr\bigg\|}{(t-s)^\beta} \nonumber\\
 &\le&
\tilde k(\rho)\|f(u(\cdot))\|_{\infty,\rho,0,T},\nonumber
\end{eqnarray}
where $\lim_{\rho\to \infty} \tilde k(\rho)=0$. In fact, we are going to show that
\begin{eqnarray}\label{trho}
\tilde k(\rho)=\bigg( \frac{1}{\rho}+c_\beta\frac{1}{\rho^{1-\beta}}+\frac{1}{\rho}T^{1-\beta}\|A_\lambda\|\bigg),
\end{eqnarray}
where $c_\beta$ is a positive constant depending on $\beta$. Note that the first term on the right hand side of (\ref{eq4}) is estimated by
\begin{eqnarray*}
\sup_{t\in[0,T]}  \int_0^te^{-\rho(t-r)}dr\|f(u(\cdot))\|_{\infty,\rho,0,T}\le \frac{1}{\rho}\|f(u(\cdot))\|_{\infty,\rho,0,T}.
\end{eqnarray*}
For the second expression,
\begin{eqnarray*}
  \frac{\int_s^te^{-\rho(t-r)}dr}{(t-s)^\beta}\le \frac{1}{\rho^{1-\beta}}\frac{1-e^{-\rho(t-s)}}{\rho^\beta(t-s)^\beta}
  \le \frac{1}{\rho^{1-\beta}}\sup_{x> 0}\frac{1-e^{-x}}{x^\beta}=:\frac{1}{\rho^{1-\beta}}c_\beta.
\end{eqnarray*}
The estimate of the last term on the right hand side of (\ref{eq4}) follows by (\ref{semi}), since $S_\lambda(t-r)-S_\lambda(s-r)=(S_\lambda(t-s)-{\rm Id})S_\lambda(s-r)$. On the other hand,
\begin{eqnarray*}
  \|f(u(\cdot))\|_{\infty,\rho,0,T}&\le& \sup_{0\le t\le T}e^{-\rho t}\|f(x)\|+\sup_{0\le t\le T}e^{-\rho t}\|f(u(t))-f(x)\|\\
 & \le &\|f(x)\|+D_fT^\beta\|u\|_{\beta,\rho,0,T},
\end{eqnarray*}
hence
\begin{eqnarray*}
 \bigg\|&\int_0^\cdot S_\lambda(\cdot-r)f(u(r))dr\bigg\|_{\beta,\rho,0,T}\le \hat k(\rho) (1+\|u\|_{\beta,\rho,0,T}),
 \end{eqnarray*}
where now $\hat k(\rho)=\max\{\|f(x)\|,D_f T^\beta \}\tilde k(\rho)$, with $\tilde k(\rho)$ defined by (\ref{trho}).

\medskip

On the other hand,
\begin{eqnarray}
  \|h(u(\cdot))\|_{\beta,\rho,0,t}&=  & \sup_{r\in[0,t]}e^{-\rho r}\|h(u(r))\|_{L_2(\ell^2)}\nonumber \\
  &+&\sup_{0\le q<r\le t}\frac{e^{-\rho r}\|h(u(r))-h(u(q))\|_{L_2(\ell^2)}}{(r-q)^\beta} \label{eq6} \\
  &\le  & \|h(x)\|+D_h(1+T^\beta)\|u\|_{\beta,\rho,0,T},\nonumber
\end{eqnarray}
hence, on account of (\ref{eh1}) we obtain
\begin{eqnarray*}
\bigg\| \int_0^\cdot S_\lambda(\cdot-r)h(u(r))d\omega(r) \bigg\|_{\beta,\rho, 0,T} \le ck(\rho)\ltn\omega\rtn_{\beta^\prime,0,T}
 (1+\|u\|_{\beta,\rho,0,T}),
\end{eqnarray*}
where $c$ may depend on $\beta$, $\beta^\prime$, $T$, $\|A_\lambda\|$, $\|h(x)\|$ and $D_h$.
In conclusion, we have obtained
\begin{eqnarray*}
 \| \tT_{x,\omega}(u)\|_{\beta,\rho,0,T}\leq  (1+\|A_\lambda\|T^{1-\beta})\|x\|+K(\rho)(1+\ltn\omega\rtn_{\beta^\prime,0,T})(1+\|u\|_{\beta,\rho,0,T})
\end{eqnarray*}
where $\lim_{\rho \to \infty}K(\rho)=0$. Note that $K(\rho)$ may also depend on the parameters related to $f$ and $h$, the initial condition $x$, $\|A_\lambda\|$ and $T$. Taking a sufficiently large $\rho$ such that $K(\rho)(1+\ltn\omega\rtn_{\beta^\prime,0,T})\leq 1/2$, the ball
$$B=B(0,R(x,\rho))=\{u\in C^\beta([0,T];\ell^2)\,:\, \|u\|_{\beta,\rho,0,T}\leq R\}$$ with $$R=R(x,\rho)=2(1+\|A_\lambda\|T^{1-\beta})\|x\|+1,$$ is mapped into itself since
\begin{eqnarray*}
  \|\tT_{x,\omega}(u)\|_{\beta,\rho,0,T}\le (1+\|A_\lambda\|T^{1-\beta})\|x\|+\frac12(1+R) =R.\\
\end{eqnarray*}
We now derive the contraction condition for the operator $\tT_{x,\omega}$ with respect to the norm $\|\cdot\|_{\beta,\bar\rho,0,T}$
where the $\bar \rho$ may differ from the $\rho$ considered above. However, since all these norms are
equivalent for different $\rho\ge 0$, the set $B$ remains a complete space with respect to any $\|\cdot\|_{\beta,\bar\rho,0,T}$.

Similar to above, for the Lebesgue integral we obtain the estimate
\begin{eqnarray*}
 \|f(u_1(\cdot))-f(u_2(\cdot))\|_{\beta,\bar\rho,0,T}
  \le \tilde k(\bar \rho) D_f\|u_1-u_2\|_{\beta,\bar\rho,0,T},
\end{eqnarray*}
where $\tilde k(\rho)$ is defined by (\ref{trho}) replacing $\rho$ by $\bar \rho$.

Regarding the stochastic integral, the difference with respect to the previous computations is that now in (\ref{eq6}) the operator $h(u(\cdot))$
has to be replaced by $h(u_1(\cdot))-h(u_2(\cdot))$. In particular, from (\ref{MN}) we easily derive
\begin{eqnarray*}
  \|h(u_1(\cdot))-h(u_2(\cdot))\|_{\beta,\bar\rho,0,T}&\le &D_h(\|u_1-u_2\|_{\infty,\bar\rho,0,T}+\sqrt{2} \ltn u_1-u_2\rtn_{\beta,\bar\rho,0,T})\\
  &+&2M_h (\|u_1\|_{\infty,0,T}+\|u_2\|_{\infty,0,T})\|u_1-u_2\|_{\infty,\bar\rho,0,T}.
\end{eqnarray*}
Since $\|u_1\|_{\infty,0,T}\leq e^{\rho T} R(x,\rho)$ (and the same inequality holds for $u_2$), then
\begin{eqnarray*}
 \| \tT_{x,\omega}(u_1)-\tT_{x,\omega}(u_2)\|_{\beta,\bar\rho,0,T}&\leq  K(\bar \rho)(1+\ltn\omega\rtn_{\beta^\prime,0,T}) (1+\|u_1\|_{\infty,0,T}+\|u_2\|_{\infty,0,T})\\
 &\qquad \times \|u_1-u_2\|_{\infty,\bar\rho,0,T},
\end{eqnarray*}
where again $\lim_{\bar \rho \to 0} K(\bar \rho)=0$. It suffices then to choose $\bar\rho$ sufficiently large so that
\begin{eqnarray*}
 \| \tT_{x,\omega}(u_1)-\tT_{x,\omega}(u_2)\|_{\beta,\bar\rho,0,T}\le \frac12 \|u_1-u_2\|_{\beta,\bar\rho,0,T},
\end{eqnarray*}
which implies the contraction property of the map $T_{x,\omega}$. Hence, (\ref{eq2}) has a unique solution $u\in C^\beta([0,T];\ell^2)$.
\end{proof}

\medskip

We finish this section by proving that the solution of (\ref{eq2}) generates a random dynamical system.\\

\begin{definition}\label{d2_1}
Consider a probability space $(\Omega,\fF,P)$. The quadruple $(\Omega,\fF,P,\theta)$ is called a metric dynamical system if the measurable mapping
\begin{eqnarray*}
    \theta:(R\times\Omega, \mathcal{B}(R)\otimes\mathcal{F}) \to
    (\Omega,\mathcal{F})
\end{eqnarray*}
is a flow, that is,
$$\theta_{t_1}\circ\theta_{t_2}=\theta_{t_1}\theta_{t_2}=\theta_{t_1+t_2},\;  t_1,\,t_2\in R; \qquad \theta_0={\rm id}_\Omega$$
and the measure $P$ is invariant and ergodic with respect to $\theta$.\\
\end{definition}
\begin{definition}\label{d2_2}
A random dynamical system $\varphi$ over the metric dynamical system
$(\Omega,\fF,P,\theta)$ is a $\big(\bB(R^+)\otimes\fF\otimes
\bB(\ell^2),\bB(\ell^2)\big)$--measurable mapping such that the cocycle
property holds
\begin{eqnarray*}
  \varphi(t+\tau,\omega,x)=\varphi(t,\theta_\tau \omega,\varphi(\tau,\omega,x)),\quad \varphi(0,\omega,x)=x,
\end{eqnarray*}
for all $t\ge\tau\in R^+$, $x\in \ell^2$ and $\omega\in \Omega$.\\
\end{definition}

The metric dynamical system is the model for the noise, in our case the fBm. More precisely, we take the quadruple $(\Omega, \fF,P,\theta)=(C_0(R;\ell^2),\bB(C_0(R;\ell^2)),P_H,\theta)$ where $\theta$ is given by the Wiener flow introduced in (\ref{shift}).

\begin{theorem}\label{t3}
The solution of (\ref{eq2}) generates a random dynamical system
$$\varphi: R^+\times \Omega\ \times \ell^2 \mapsto \ell^2$$
given by $\varphi(t,\omega,x)=u(t,\omega,x)=u(t)$, where $u$ the unique solution to (\ref{eq2}) corresponding to $\omega$ and initial condition $x$.
\end{theorem}
\begin{proof}
We only sketch the main ideas of the proof.

The cocycle property is a consequence in particular of the additivity of the stochastic integral as well as the behavior of the stochastic integral when performing a change of variable given by (\ref{change}). More specifically,
\begin{eqnarray*}
\varphi (t+\tau, \omega,x) &= &
S_\lambda(t+\tau)x+\int_{0}^{t+\tau}S_\lambda(t+\tau-r)
f(u(r))dr\\
&\qquad +&\int_{0}^{t+\tau}S_\lambda(t+\tau-r)
h(u(r))d\omega(r) \\
&=&S_\lambda(t) \left ( S_\lambda(\tau)x+\int_{0}^{\tau}S_\lambda(\tau-r)f(u(r))dr+\int_{0}^{\tau}S_\lambda(\tau-r)h(u(r))d\omega(r) \right) \\
& \qquad +&\int_{\tau}^{t+\tau}S_\lambda(t+\tau-r)f(u(r))dr+\int_{\tau}^{t+\tau}S_\lambda(t+\tau-r)h(u(r))d\omega(r)\\
&=&S(t)u(\tau)+\int_{0}^{t}S_\lambda(t-r)f(u(r+\tau))dr+\int_{0}^{t}S_\lambda(t-r)h(u(r+\tau))d\theta_\tau\omega(r).
\end{eqnarray*}
Denoting $y(\cdot)=u(\cdot+\tau)$ the previous inequality reads
$$\varphi  (t+\tau, \omega,x)=S_\lambda(t)y(0)+\int_{0}^{t}S_\lambda(t-r)f(y(r))dr+\int_{0}^{t}S_\lambda(t-r)h(y(r))d\theta_\tau\omega(r)$$
and the right hand side of the last equality is equal to $\varphi(t,\theta_\tau \omega, \varphi(\tau,\omega,x))$.

The measurability of the mapping $\varphi$ follows due to its continuity with respect to $\omega$ (that implies measurability with respect to $\omega$), also due to its continuity with respect to $(t,x)$ and the separability of $\ell^2$, see Lemma
III.14 in Castaing and Valadier \cite{CasVal77}.

\end{proof}

\section{Exponential stability of the trivial solution}\label{s3}

The purpose of this section is to show that the trivial solution of (\ref{eq0}) is exponential stable. Therefore, we start assuming that zero is an equilibrium of the SLDS.\\

Since we work directly with our SLDS without transforming it into a random equation, the norm of the solution depends on how large is the norm of the noisy input, and therefore we will consider a cut--off strategy, in such a way that we will deal with a modified lattice system depending on a random variable. Further that random variable can be chosen in a suitable way such that it turns out that it is possible to apply a Gronwall--like lemma, that together with the temperedness of the involved random variables will imply that the solution of the modified system coincides with the one of the original lattice system, that converges to the trivial solution exponentially fast.\\

\begin{definition}\label{d1}
The trivial solution of the SLDS is said to be exponential stable with rate $\mu>0$ if for almost every $\omega$ there exists a random variable $\alpha(\omega)>0$ and a random neighborhood $U(\omega)$ of zero
such that for all $\omega \in \Omega$ and $t\in R^+$
\begin{eqnarray*}
\sup_{x\in U(\omega)} \|\varphi(t,\omega,x) \|\leq \alpha (\omega) e^{-\mu t},
\end{eqnarray*}
where $\varphi:R^+ \times \Omega \times \ell^2 \rightarrow \ell^2$ is the cocycle mapping given in Theorem \ref{t3}.\\
\end{definition}

For the study of exponential stability of systems driven by continuous semimartingales, see the monograph \cite{Mao}. In the spirit of working with the rich theory of RDS, here we have adapted the definition of exponential stability to the RDS setting.

We would like to prove that the trivial solution of the SLDS is exponentially stable with rate $\mu<\lambda$. In order to do that, first of all we need to introduce the key concept of temperedness. A random variable $R\in (0,\infty)$ is called tempered from above with respect to the metric dynamical system $(\Omega,\fF,P,\theta)$ if
\begin{eqnarray}\label{eq4b}
\limsup_{t\to \pm\infty}\frac{\log^+R(\theta_t\omega)}{t}=0\quad {\rm with \; probability \; 1}.
\end{eqnarray}
Therefore, temperedness from above describes the subexponential growth  of a stochastic stationary process
$(t,\omega)\mapsto R(\theta_t\omega)$.
$R$ is called tempered from below if $R^{-1}$ is tempered from above. In particular, if the random variable $R$ is tempered from below and
$t\mapsto R(\theta_t\omega)$ is continuous, then for any $\eps>0$ there exists a random variable $C_\eps(\omega)>0$ such that
\begin{eqnarray*}
  R(\theta_t\omega)\ge C_\eps(\omega)e^{-\eps|t|}\quad {\rm with \; probability \; 1}.
\end{eqnarray*}
A sufficient condition for temperedness with respect to an ergodic metric dynamical system  is that
\begin{eqnarray*}
  E\sup_{t\in [0,1]}\log^+R(\theta_t\omega)<\infty,
\end{eqnarray*}
see Arnold \cite{Arnold}, Page 165. Hence, by Kunita \cite{Kun90} Theorem 1.4.1 we obtain that $R(\omega)=\ltn\omega\rtn_{\beta^\prime,0,1}$ is tempered from above because $\log^+r\le r$ for $r> 0$ and trivially $\sup_{t\in [0,1]}\ltn\theta_t\omega\rtn_{\beta,0,1}\le \ltn\omega\rtn_{\beta,0,2}$. Furthermore, the set of all $\omega$ satisfying (\ref{eq4b}) is invariant with respect to the flow $\theta$.
\medskip

We now introduce two more assumptions, which in particular imply that (\ref{eq0}) has the unique trivial solution.\\

In what follows, for $\delta>0$ we also assume that
\begin{itemize}
\item [(\textbf{A3}')]
Each $f_i$ is defined on $[-\delta,\delta]$. In addition to the assumption (\textbf{A3}), we assume that $f_i(0)=f_i^\prime(0)=0$, $f_i\in C^2([-\delta,\delta]),R)$ and there exists a positive constant $M_f$ such that
\begin{eqnarray*}
|f^{\prime\prime}_i(\zeta)|\leq M_f,\quad \zeta\in [-\delta, \delta], \, i\in Z.
\end{eqnarray*}
\item [(\textbf{A4}')]
Let each $h_i$ be defined on $[-\delta,\delta]$.
In addition to the assumption (\textbf{A4}), we assume that $h_i(0)=h_i^\prime(0)=0$.
\end{itemize}
\medskip
The operators $f,\,h$ then are defined on $\bar B_{\ell^2}(0,\delta)$. In particular, from {\bf (A3')} we derive that $f$ is Fr\'echet differentiable and its derivative $Df:\ell^2 \mapsto L(\ell^2)$ is continuous. Indeed, for $u,\,v\in \ell^2$ we obtain
$$\|f(u+v)-f(u)-Df(u)v\|^2\leq \frac14 M_f^2\|v\|^4,$$
and
$$\|Df(u)-Df(v)\|^2_{L(\ell^2)} =\sup_{\|z\|=1}\|Df(u)z-Df(v)z\|^2\leq M_f^2\|u-v\|^2.$$

Furthermore, these assumptions ensure that (\ref{eq0}) has the unique trivial solution.

We introduce $\chi$ to be the cut--off function
\begin{eqnarray*}
  \chi:\ell^2\to \bar B_{\ell^2}(0,1),\quad\chi(u)=\left\{\begin{array}{lcr}
  u&:& \|u\|\le \frac12\\
  0&:& \|u\|\ge 1
  \end{array}
   \right.
\end{eqnarray*}
such that the norm of $\chi(u)$ is bounded by 1. We also assume that $\chi$ is twice continuously differentiable with bounded derivatives $D\chi$ and $D^2 \chi$. Bounds of these derivatives are denoted by $L_{D\chi},\,L_{D^2\chi}$. Now for $u\in \ell^2$ and some $0<\hat R\le \delta$ we define
$$\chi_{\hat R}(u)=\hat R\chi(u/\hat R)\in \bar B_{\ell^2}(0,\hat R).$$
Then it is easy to see that the first derivative $D\chi_{\hat R}$ of $\chi_{\hat R}$ is bounded by $L_{D\chi}$, while the second derivative $D^2\chi_{\hat R}$ is bounded by$\frac{L_{D^2\chi}}{\hat R}$.

We now modify the operators $f,\,h$ by considering their compositions with the above cut--off function. In that way, we set $f_{\hat R}:=f\circ \chi_{\hat R}:\ell^2 \to \ell^2$ and $h_{\hat R}:=h\circ \chi_{\hat R}:\ell^2 \to L_2(\ell^2)$, consider (\ref{eq2})
replacing $f$ by $f_{\hat R}$ and $h$ by $h_{\hat R}$, and the sequence $(u^n)_{n\in N}$ defined by
\begin{eqnarray}
u^n(t)&=&S_\lambda(t)u^n(0)+\int_0^t S_\lambda (t-r)f_{\hat R(\theta_n \omega)}(u^n(r))dr\nonumber \\[-1.5ex]
\label{eq2b}\\[-1.5ex]
&+&\int_0^t S_\lambda (t-r)h_{\hat R(\theta_n \omega)}(u^n(r))d\theta_n\omega,\qquad t\in [0,1],\nonumber
\end{eqnarray}
where $u^0(0)=x$ and $u^n(0)=u^{n-1}(1)$. Since the modified coefficients satisfy assumptions in Theorem \ref{t2} for any $n\in N$, then there exists a unique solution $u^n$ to (\ref{eq2b}) on $[0,1]$.

Next we establish a result which will be key in order to obtain the exponential stability of the trivial solution.\\

\begin{lemma}\label{l51}
For every $R>0$ there exists a positive $\hat R\le \delta$ such that for all $u, \, z\in \ell^2$
\begin{eqnarray}
  \|f_{\hat R}(u)\|&\le& R L_{D\chi}\|u\|,\label{eq13} \\
  \|h_{\hat R}(u)\|&\le& R L_{D\chi} \|u\|, \label{eq9} \\
   \|h_{\hat R}(u)-h_{\hat R}(z)\|&\le& R L_{D\chi}\|u-z\|.\label{eq23b}
\end{eqnarray}

\end{lemma}
\begin{proof}
By $Df(0)=0$ and the continuity of $Df$, for any $R>0$ we can choose an $\hat R\le \delta$ such that
\begin{eqnarray*}
 \sup_{\|v\|\le \hat R}\|Df(v)\|_{L(\ell^2)}\le  R.
\end{eqnarray*}
Then for $u\in \ell^2$, since $f(0)=0$ from the mean value theorem we have
\begin{eqnarray*}
  \|f_{\hat R}(u)\|&\le \sup_{z\in \ell^2}\|D(f(\chi_{\hat R}(z)))\| \|u\| \le  \sup_{\|v\|\le \hat R}\|Df(v)\|_{L(\ell^2)}\sup_{z\in \ell^2}\|D\chi_{\hat R}(z)\|\|u\|\\
&\le R L_{D\chi}\|u\|,
  \end{eqnarray*}
and therefore (\ref{eq13}) is shown. Following the same steps we prove (\ref{eq9}).

Finally, by the regularity of $Dh$,
\begin{eqnarray*}
\|h_{\hat R}(u)-h_{\hat R}(z)\|& \le \sup_{\|v\|\le \hat R}\|Dh(v)\|_{L(\ell^2,L_2(\ell^2))}\| \chi_{\hat R} (u)-\chi_{\hat R} (z)\|\\
& \le L_{D\chi} \sup_{\|v\|\leq \hat R}\|Dh(v)\|_{L(\ell^2,L_2(\ell^2))} \|u-z\| \leq R L_{D\chi} \|u-z\|.
\end{eqnarray*}
\end{proof}

Note that, according to the proof of the Lemma \ref{l51}, the relationship between $R$, $\hat R$ and $\delta$ is given by
\begin{eqnarray*}
  &\hat R(\omega)=\max\bigg\{\hat r:\sup_{\|v\| \leq \hat r}(\|Df(v)\|+\|Dh(v)\|)\le  R(\omega)\bigg\}\wedge \delta.
\end{eqnarray*}
Hence, once that we will define $R$, this will be further the precise definition of $\hat R$ in order to ensure exponential stability of the trivial solution of the stochastic lattice model, see (\ref{Rhat}) below.\\

For $n\in Z^+$, we set
\begin{equation}\label{eq8}
  u(t)=u^n(t-n)\quad {\rm if \; }t\in [n,n+1].
\end{equation}
Let us emphasize that the previous function $u$ is defined on the whole positive real line and is H\"older continuous on any interval $[n,n+1]$. However, we cannot claim yet that  $u$ defined by (\ref{eq8}) is our mild solution obtained in Theorem \ref{t2}. The reason is that any $u^n$ is a solution of a modified lattice problem depending of the cut--off function $\chi_{\hat R}$ and driven by a path $\theta_n\omega$. But as we will show below, using the additivity of the integrals, the estimates of the functions $f_{\hat R}$ and $h_{\hat R}$ given in Lemma \ref{l51}, and a suitable choice of the random variables $R$ and $\hat R$, we will end up proving that not only $u$ given by (\ref{eq8}) is the solution of our original stochastic lattice system (\ref{eq2}), but also that it is locally exponential stable with a certain decay rate $\mu$.

In order to prove the previous assertions, we first express $u$ given by (\ref{eq8}), for $t\in [n,n+1]$ as follows
\begin{eqnarray}
 u(t) &=& S_\lambda(t-n)u(n) +\int_n^t S_\lambda(t-r)f_{\hat R(\theta_n\omega)}(u(r))dr+\int_n^tS_\lambda(t-r) h_{\hat R(\theta_n\omega)}(u(r))d\omega(r)\nonumber \\
 & =&  S_\lambda (t)x+\sum_{j=0}^{n-1}S_\lambda(t-j-1) \bigg(\int_{j}^{j+1}S_\lambda(j+1-r)f_{\hat R(\theta_j\omega)}(u(r))dr \nonumber \\
 &&\qquad \qquad+\int_{j}^{j+1}S_\lambda(j+1-r)h_{\hat R(\theta_j\omega)}(u(r))d\omega(r)  \bigg)\nonumber \\
  &&\qquad \qquad+\int_{n}^{t}S_\lambda(t-r)f_{\hat R(\theta_n\omega)}(u(r))dr+\int_n^{t}S_\lambda(t-r)h_{\hat R(\theta_n\omega)}(u(r))d\omega(r) \label{eq10} \\
  &=&S_\lambda (t) x+\sum_{j=0}^{n-1}S_\lambda(t-j-1) \bigg(\int_0^1 S_\lambda(1-r)f_{\hat R(\theta_j\omega)}(u^j(r))dr\nonumber \\
  &&\qquad \qquad+\int_0^1S_\lambda(1-r)h_{\hat R(\theta_j\omega)}(u^j(r))d\theta_j\omega(r)  \bigg)\nonumber \\
  &+&\int_{0}^{t-n}S_\lambda(t-n-r)f_{\hat R(\theta_n\omega)}(u^n(r))dr+\int_0^{t-n}S_\lambda(t-n-r) h_{\hat R(\theta_n\omega)}(u^n(r))d\theta_n\omega(r),\nonumber
\end{eqnarray}
where this splitting is a consequence of the additivity of the integrals, Theorem \ref{t1} and (\ref{change}).

Notice that, in all the integrals on the right hand side of the previous expression, the time varies in the interval $[0,1]$ (in the last two integrals, $[0,t-n]$ is contained in $[0,1]$). Hence, we are going to estimate the H{\"o}lder--norm of all these terms setting now $T_1=0$ and $T_2=1$. Due to the presence of the semigroup $S_\lambda$ as a factor in all terms under the sum, in the following estimates we do not need to consider the $\beta,\rho$--norm but the $\beta$--norm, that is, in what follows $\rho=0$.

Note that by (\ref{eq13}) we have
\begin{eqnarray*}
\bigg\|\int_0^\cdot S_\lambda(\cdot-r)f_{\hat R(\theta_n\omega)}(u^n(r))dr\bigg\|_{\infty}\le R(\theta_n \omega) L_{D\chi} \|u^n\|_{\infty}.
  \end{eqnarray*}
For the H{\"o}lder--seminorm, thanks to (\ref{semi}),
\begin{eqnarray*}
&  &\ltn\int_0^\cdot S_\lambda(\cdot-r) f_{\hat R(\theta_n\omega)}(u^n(r))dr\rtn_{\beta}\\
 & =&\sup_{0\leq s<t\leq 1} \frac{\bigg \|\int_s^tS_\lambda(t-r)f_{\hat R(\theta_n\omega)}(u^n(r))dr+\int_0^s (S_\lambda(t-r)-S_\lambda(s-r))f_{\hat R(\theta_n\omega)}(u^n(r))dr\bigg\|}{(t-s)^{\beta}}\\
  &\le& \sup_{0\leq s<t\leq 1}  \bigg((t-s)^{1-\beta}\sup_{r\in[s,t]}(\|S_\lambda(t-r)\|_{L(\ell^2)}\|f_{\hat R(\theta_n\omega)}(u^n(r))\|)\bigg)\\
  &\,\, +&\sup_{0\leq s<t\leq 1} \bigg(\frac{s}{(t-s)^\beta} \sup_{r\in[0,s]}(\|S_\lambda(t-r)-S_\lambda(s-r)\|_{L(\ell^2)}\|f_{\hat R(\theta_n\omega)}(u^n(r))\|)\bigg)\\
&  \le &R(\theta_n\omega) L_{D\chi}\|u^n\|_{\infty}
 + \|A_\lambda\|R(\theta_n\omega) L_{D\chi} \|u^n\|_{\infty}.
 \end{eqnarray*}
Then
\begin{eqnarray*}
 \bigg \|\int_0^\cdot   & S_\lambda (\cdot-r) f_{\hat R(\theta_n\omega)}(u^n(r))dr \bigg \|_{\beta} \leq (2+\|A_\lambda\|)R(\theta_n\omega) L_{D\chi}\|u^n\|_{\beta}.
\end{eqnarray*}
On the other hand, since $h(0)=0$, by (\ref{eq9}) and (\ref{eq23b}) we get
\begin{eqnarray*}
  \|h_{\hat R(\theta_n\omega)}(u(\cdot))\|_\beta&=&\sup_{t\in[0,1]}\|h_{\hat R(\theta_n\omega)}(u(t))\|\\
  &+&
  \sup_{0\le r<q\le 1}\frac{\|h_{\hat R(\theta_n\omega)}(u(r))-h_{\hat R(\theta_n\omega)}(u(q))\|}{(r-q)^\beta}\\
 & \le &  L_{D\chi} R(\theta_n\omega)\|u\|_\beta.
\end{eqnarray*}
For the stochastic integral, thanks to (\ref{eh2}) and (\ref{eh3}) we obtain
\begin{eqnarray*}
 & & \bigg\|\int_0^{\cdot} S_\lambda(\cdot-r) h_{\hat R(\theta_n\omega)}(u^n(r))d\theta_n\omega(r)\bigg\|_\beta\\
  &\le &c_{\beta,\beta^\prime}\ltn\theta_n\omega\rtn_{\beta^\prime}
  (1+\|A_\lambda\|)(2+\|A_\lambda\|)\|h_{\hat R(\theta_n\omega)}(u^n(\cdot))\|_\beta\\
  &\leq&  L_{D\chi}c_{\beta,\beta^\prime}\ltn\theta_n\omega\rtn_{\beta^\prime}
 (1+\|A_\lambda\|)(2+\|A_\lambda\|)R(\theta_n\omega)\|u^n\|_\beta.
  \end{eqnarray*}

 For the terms under the sum we have
\begin{eqnarray*}
 & & \bigg\| S_\lambda(\cdot-j-1)\int_{0}^{1}S_\lambda(1-r)f_{\hat R(\theta_j\omega)}(u^j(r))dr \bigg \|_{\beta,n,n+1}\\
 &=&\| S_\lambda(\cdot-j-1)\|_{\beta,n,n+1} \bigg\|\int_{0}^{1}S_\lambda(1-r) f_{\hat R(\theta_j\omega)}(u^j(r))dr\bigg \|\\
   &\leq & \|S_\lambda(\cdot-j-1) \|_{\beta,n,n+1} \bigg\| \int_{0}^{\cdot}S_\lambda(\cdot-r) f_{\hat R(\theta_j\omega)}(u^j(r))dr\bigg \|_{\infty},
\end{eqnarray*}
and from Lemma \ref{l3},
$$ \|S_\lambda(\cdot-j-1) \|_{\beta,n,n+1} \leq (1+\|A_\lambda\|)e^{-\lambda(n-j-1)},$$
such that
\begin{eqnarray*}
 \bigg\| & S_\lambda(\cdot-j-1)\int_{0}^{1}S_\lambda(1-r)f_{\hat R(\theta_j\omega)}(u^j(r))dr \bigg \|_{\beta,n,n+1}\\
& \leq (1+\|A_\lambda\|) e^{-\lambda(n-j-1)}  R(\theta_j\omega) L_{D\chi}\|u^j\|_{\beta}.
\end{eqnarray*}

Following similar steps, thanks to (\ref{eh2}) we have
\begin{eqnarray*}
 \bigg\| &S_\lambda(\cdot-j-1) \int_{0}^{1}S_\lambda(1-r) h_{\hat R(\theta_j\omega)}(u^j(r))d\theta_j\omega(r)\bigg \|_{\beta,n,n+1}\\
 &\leq L_{D\chi} c_{\beta,\beta^\prime} \ltn \theta_j\omega \rtn_{\beta^\prime} (1+\|A_\lambda\|)^2 e^{-\lambda(n-j-1)}    R(\theta_j\omega) \|u^j\|_{\beta}.
  \end{eqnarray*}

Therefore, taking the $\|\cdot\|_{\beta,n,n+1}$ norm of the different terms in (\ref{eq10}), applying the triangle inequality and in view of the above estimates, we obtain
\begin{eqnarray}
  \|u^n\|_{\beta}&\le &   \|S_\lambda\|_{\beta,n,n+1} \|x\| +C\sum_{j=0}^{n-1}
   R(\theta_j\omega)(1+\ltn \theta_j\omega \rtn_{\beta^\prime})\|u^j\|_{\beta}e^{-\lambda(n-j-1)}\nonumber \\[-1.5ex]
   \label{esti}\\[-1.5ex]
   &+ &C R(\theta_n\omega) (1+\ltn \theta_n\omega \rtn_{\beta^\prime}) \|u^n\|_{\beta},\nonumber
\end{eqnarray}
where $C=\max \{1,c_{\beta,\beta^\prime} \} L_{D\chi}(1+\|A_\lambda\|)(2+\|A_\lambda\|).$

Let now $\hat\eps \in (0,1)$, that will be determined later more precisely. Define the variables $R$ and $\hat R$ as follows:
\begin{eqnarray*}
  R(\omega)=\frac{\hat \eps}{2C(1+\ltn \omega \rtn_{\beta^\prime})}
\end{eqnarray*}
and
\begin{eqnarray}\label{Rhat}
  &\hat R(\omega)=\max\bigg\{\hat r:\sup_{\|v\| \leq \hat r}(\|Df(v)\|+\|Dh(v)\|)\le  R(\omega)\bigg\}\wedge \delta.
\end{eqnarray}
$\hat R(\omega)$ is a random variable, see \cite{GaNeSch16}. In addition, since $\ltn\omega \rtn_{\beta^\prime}$ is tempered from above then $R$ is tempered from below.
According to Lemma \ref{l21} it follows that $\hat R$ is tempered. In the contrary case we have  $\omega\in\Omega$, $\mu\in R^+\setminus \{0\}\cup \{+\infty\}$ and a sequence $(t_i)_{i\in N}$ tending to $+\infty$ or $-\infty$ such that
\[
  \hat R(\theta_{t_i}\omega)\le e^{-\mu|t_i|}.\nonumber
\]
But then for sufficiently large $i$ we have
\[
  R(\theta_{t_i}\omega)\le \frac{1}{\kappa}e^{-\mu|t_i|}\nonumber
\]
contradicting the temperedness of $R$
 (here we have applied Lemma \ref{l21}, with  $F(\hat r):= \sup_{\|v\| \leq \hat r}(\|Df(v)\|+\|Dh(v)\|)\le  R(\omega)$. Notice that the assumptions of that lemma are fulfilled thanks to the regularity properties of the functions $f$ and $h$. In particular, we need any $f_i$ to be two times differentiable).

With the above choice of $R$,  coming back to (\ref{esti}), since $\hat \epsilon <1$ we obtain
\begin{eqnarray*}
  \frac{1}{2}\|u^n\|_{\beta}\le &  \|S_\lambda\|_{\beta,n,n+1} \|x\|  +\frac{\hat \eps}{2} \sum_{j=0}^{n-1} e^{-\lambda (n-j-1)} \|u^j\|_{\beta},
\end{eqnarray*}
hence
\begin{eqnarray*}
\|u^n\|_{\beta}  \le& 2 (1+\|A_\lambda\|) \|x\|e^{-\lambda n}+\hat\eps \sum_{j=0}^{n-1}e^{-\lambda(n-j-1)}  \|u^j\|_{\beta}.
  \end{eqnarray*}
Defining $y_n=\|u^n\|_\beta e^{\lambda n}$, $c= 2 (1+\|A_\lambda\|)$ and $g_j=\hat \eps e^\lambda$, Lemma \ref{l9} ensures that
\begin{eqnarray*}
y_n \le  2(1+\|A_\lambda\|) \|x\| \prod_{i=0}^{n-1} (1+ \hat \eps e^{\lambda})=2(1+\|A_\lambda\|) \|x\| (1+ \hat \eps e^{\lambda})^n,
\end{eqnarray*}
hence
\begin{eqnarray}\label{eun}
\|u^n\|_{\beta}\le 2(1+\|A_\lambda\|) \|x\| e^{-n(\lambda-\log(1+\hat\eps e^{\lambda}))}.
\end{eqnarray}

On the other hand, due to Lemma \ref{l8}, since $\hat R(\omega)/2$ is tempered from below, we can find a zero neighborhood $U$ depending on $\omega$  such that for $x$ contained in this neighborhood we have
\begin{equation}\label{enm}
  \|u^n\|_{\beta}\le \frac{\hat R(\theta_n\omega)}{2}\quad {\rm for \; all\; } n\in Z^+,
\end{equation}

\smallskip

As we will show in the next result, (\ref{enm}) is a crucial estimate to prove that the sequence of truncated solutions $(u^n)_{n\in N}$ defines a solution of (\ref{eq2}) on $R^+$. In fact, thanks to the previous considerations, we can state the main result of the paper:\\

\begin{theorem}\label{t4}
Suppose that conditions ${\bf (A1)-(A4)}$ and ${\bf (A3')-(A4')}$ hold and consider $\hat \eps (\lambda)=\hat \eps \in (0,1-e^{-\lambda})$.  Then the trivial solution is exponentially stable with an exponential rate less than or equal to $\mu <\lambda-\log(1+\hat\eps e^\lambda)$.\\
\end{theorem}

\begin{proof}
First of all, let us prove that $u$ given by (\ref{eq8}) is a solution of our original stochastic lattice system. In fact, from (\ref{enm}) we deduce that for any $j\in Z^+$ and any $r\in [0,1]$ we have that
\begin{eqnarray*}
\|u^j(r)\|\leq \frac{\hat R(\theta_j\omega) }{2}.
\end{eqnarray*}
Consequently $\chi_{\hat R(\theta_j \omega)} (u^j (r))=u^j (r)$, hence
\begin{eqnarray*}
  f_{\hat R(\theta_j\omega)}(u^j(r))=f(u^j(r)),\quad h_{\hat R(\theta_j\omega)}(u^j(r))=h(u^j(r))
\end{eqnarray*}
for any $r\in [0,1]$ and $j\in Z^+$.  Then $u$ given by (\ref{eq8}), where $u^n$ solves (\ref{eq2b}), is a solution of (\ref{eq2}) on $R^+$.\\

Now we show the exponential stability of the trivial solution, according to the Definition \ref{d1}. Take $\hat \eps \in (0,1-e^{-\lambda})$, $\eps>0$ small enough and $\mu <\lambda-\log(1+\hat\eps e^\lambda)-\eps$.  From (\ref{eun}) we derive that there exist $T_0(\omega, \eps)\in N$ and a neighborhood of zero $U(\omega)$ such that if $x\in U(\omega)$, for $t\geq T_0(\omega,\eps)$
 \begin{eqnarray*}
\|\varphi(t,\omega,x)\| &\leq &\alpha_1(\omega) e^{-(\lambda-\log(1+\hat\eps e^{\lambda})-\eps)t}\leq \alpha_1(\omega) e^{-\mu t},
\end{eqnarray*}
where the positive random variable $\alpha_1(\omega)$ depends on the coefficients $\kappa$ and $\|A_\lambda\|$.

On the other hand, from (\ref{enm}) we derive
 \begin{eqnarray*}
\sup_{x\in U(\omega), t\in [0,T_0(\omega,\eps)]}\|\varphi(t,\omega,x)\|&\le& \sup_{x\in U(\omega), 0\le n\le T_0(\omega,\eps)}\|u^n\|_\beta \leq   \alpha_2(\omega) e^{-\mu t} ,
\end{eqnarray*}
where
$$\alpha_2(\omega)=\sup_{0\le n\le T_0(\omega,\eps)} \frac{\hat R(\theta_n\omega)}{2} e^{\mu T_0} .$$
Therefore,
 \begin{eqnarray*}
\sup_{x\in U(\omega)}\|\varphi(t,\omega,x)\|&\le& \sup_{x\in U(\omega),t\in [0,T_0(\omega,\eps)]}\|\varphi(t,\omega,x)\|+\sup_{x\in U(\omega), t\geq T_0(\omega,\eps)}\|\varphi(t,\omega,x)\| \\
&\leq & (\alpha_1(\omega)+\alpha_2(\omega)) e^{-\mu t}.
\end{eqnarray*}
\end{proof}

\section{Appendix}
In this section we present some technical results that we have used in Section \ref{s3}.

First of all, we introduce a discrete Gronwall-like lemma, whose proof can be derived easily from Lemma 100 in \cite{Dra}.
\begin{lemma}\label{l9}
Let $(y_n)$ and $(g_n)$ be nonnegative sequences and $c$ a nonnegative constant. If
$$y_n\leq c+ \sum_{j=0}^{n-1}  g_j y_j$$
then
$$y_n \leq c \prod_{j=0}^{n-1}  (1+g_j).$$

\end{lemma}

\begin{lemma}\label{l8}
Suppose that  $R_i\ge C_\eps e^{-\eps i}$ for any $0< \eps< \mu$ and $i\in N$, where $C_\eps>0$.
Let $(v_i)_{i\in N}$ be a sequence such that $v_i\le  v_0e^{-\mu i}$.
Then for sufficiently small $v_0$ we have for any $i\in N$
\begin{eqnarray*}
  v_i\le R_i.
\end{eqnarray*}
\end{lemma}
The proof of this result follows easily. Note that if for instance we assume that a random variable $R>0$ is tempered from below, then we can find a random variable $C_\eps>0$ such that $v_i<R(\theta_i\omega)$ holds for $v_0<C_\eps(\omega)$.\\

The following result establishes the relationship between the random variables $R$ and $\hat R$ as needed in Section \ref{s3}.\\
\begin{lemma}\label{l21}
Let $(V,\|\cdot\|)$ be some Banach space and let $F\not\equiv0$ be a function from $\bar B(0,\rho)\subset V$ into  $V$ with $F(0)=0$ which is continuously differentiable   such that
\begin{eqnarray*}
 \sup_{z\in \bar B(0,\rho)}\|DF(z)\|=\kappa<\infty.
\end{eqnarray*}
Consider the  centered open ball $B(0,R)$, $R>0,$ in $V$.   Let $B(0,\hat R)\subset V$, $\hat R=\hat R(R)\le\rho$  be the
supremum of all numbers $\hat r>0$ such that
\begin{eqnarray*}
  B(0,{\hat r})\subset F^{-1}(B(0,R)).
\end{eqnarray*}
Then, for $0\leq R < \sup\{\|F(z)\|,\, z\in\bar B(0,\rho)\}$,
\begin{eqnarray}\label{eq27}
 \sup_{z\in \bar B(0,\hat R(R))}\|F(z)\|\le R,\qquad  \frac{\hat R(R)}{R}\ge \frac{1}{\kappa}\in (0,\infty].
\end{eqnarray}
\end{lemma}
\begin{proof}
Denote
\begin{eqnarray*}
  f_F:\bar B(0,\rho)\to R^+,\,f_F(z)=\|F(z)\|, \, z\in \bar B(0,\rho).
\end{eqnarray*}
Let us define
\begin{eqnarray*}
  \hat R=\sup\{0<\hat r\le \rho: B(0,\hat r)\cap f_F^{-1}(\{R\})=\emptyset\}.
\end{eqnarray*}
Note that $f_F^{-1}(\{R\})\not =\emptyset$ since $R\in f_F(\bar B(0,\rho))$. Moreover, the set defining $\hat R$ is nonempty since
$$f^{-1}_F([0,R)) \cap f^{-1}_F(\{R\})=\emptyset, \quad 0\in f^{-1}_F([0,R)),$$
therefore by the continuity of $f_F$ there exists always a positive $\hat r$ such that $B(0,\hat r)\subset f^{-1}_F([0,R))$.

On the other hand, the ball $B(0,\hat R)$ does not contain a $\hat z$ such that $R<f_F(\hat z)=:R_1$.
In the other case, by the continuity of $f_F$, the set $f_F(B(0,\|\hat z\|))$ would contain the interval $[0,R_1)$ which includes $R$ and there would exist a $\hat{\hat z} \in V$ with
\begin{eqnarray*}
  \|\hat{\hat z}\|\le \|\hat z\|<\hat R ,\quad f_F(\hat{\hat z})=R
\end{eqnarray*}
which contradicts the definition of $\hat R$. Note that by the connectedness of balls their images by a continuous function are intervals. Hence
\begin{eqnarray*}
  B(0,\hat R)\cap f_F^{-1}((R,\infty))=\emptyset \quad {\rm and }\quad \bar B(0,\hat R)\cap f_F^{-1}((R,\infty))=\emptyset
\end{eqnarray*}
which proves the first part of (\ref{eq27}).\\
Furthermore, by the definition of $\hat R$ for every $\eps>0$ sufficiently small there exist $x_\eps^R\in B(0,\hat R)$, with $\inf_{\eps>0} \|x_\eps^R\| >0$, and $y_\eps^R\in f_F^{-1}(\{R\})$ such that $\|x_\eps^R-y_\eps^R\|<\eps$.
Hence
\begin{eqnarray*}
  \frac{\hat R}{R}&=&\lim_{\eps\to 0}\frac{\|x_\eps^R\|}{f_F(y_\eps^R)} \ge \liminf_{\eps\to 0}\frac{\|x_\eps^R\|}{f_F(x_\eps^R)+\sup_{z\in \bar B(0,\rho)}\|DF(z)\|\eps}\\
 & \ge&\lim_{\eps\to 0}\frac{\|x_\eps^R\|}{\sup_{z\in \bar B(0,\rho)}\|DF(z)\|(\|x_\eps^R\|+\eps)}=\frac{1}{\kappa}\lim_{\eps\to 0}\frac{\|x_\eps^R\|}{\|x_\eps^R\|+\eps}=\frac{1}{\kappa}
\end{eqnarray*}
since by Taylor's formula
\begin{eqnarray*}
   f_F(y_\eps^R)&\le& f_F(x_\eps^R)+\|F(x_\eps^R)-F(y_\eps^R)\|\le f_F(x_\eps^R)+\sup_{z\in \bar B(0,\rho)}\|DF(z)\|\eps,\\
   f_F(x_\eps^R)&\le& \sup_{z\in \bar B(0,\rho)}\|DF(z)\|\|x_\eps^R\|.
\end{eqnarray*}

\end{proof}

Note that it is not necessary to consider $F\equiv 0$, because the results derived from the last lemma follow trivially.\\

\bibliography{lattice}\bibliographystyle{plain}

\end{document}